\documentclass[a4paper]{article}\usepackage[english]{babel}

\usepackage[utf8]{inputenc}
\usepackage[english]{babel}
\usepackage{amsthm}
\usepackage{amssymb}
\usepackage{authblk}
\usepackage{mathtools}
\usepackage{xcolor}
\usepackage{hyperref}
\usepackage{graphicx} % Required for inserting images
\usepackage{csquotes}

\usepackage[backend=biber, giveninits=true, maxnames=99, minnames=99, url=false, doi=false]{biblatex}%Imports biblatex package
\usepackage{titling}
\newcommand{\subtitle}[1]{%
  \posttitle{%
    \par\end{center}
    \begin{center}\large#1\end{center}
    \vskip0.5em}%
}
\numberwithin{equation}{section}

\hypersetup{
    colorlinks=true,
    linkcolor=blue,
    filecolor=magenta,      
    urlcolor=cyan,
    pdftitle={Overleaf Example},
    pdfpagemode=FullScreen,
    }

\newtheorem{theorem}{Theorem}[section]

\newtheorem{proposition}{Proposition}[section]
\newtheorem{lemma}{Lemma}[section]
\newtheorem{definition}{Definition}[section]

\theoremstyle{remark}
\newtheorem*{remark}{Remark}

\newcommand{\ue}{u_{\epsilon}}

\title{BOUNDED TIME INVERSE SCATTERING FOR SEMILINEAR DIRAC EQUATION}

\author{Yuchao Yi\thanks{Department of Mathematics, University of California San Diego, La Jolla, California, United States of America; Email: yuyi@ucsd.edu}}
\date{\vspace{-5ex}}
\begin{document}

\maketitle

\begin{abstract}
    In this paper, we present the first uniqueness result on the bounded time inverse scattering problem for a semilinear Dirac equation with smooth nonlinearity $F(x, z)$ where $(x, z)\in \mathbb{R}^3\times \mathbb{C}^4$ and $x$ is the spatial variable. We show that the solution map, which sends initial data at time 0 to the solution at time $T$, uniquely determines $F(x, z)$ on $x \in \mathbb{R}^3$ and $|z| \leq M$, where $M$ is a constant depend on the solution map, under the assumption that $\partial_z F(x, 0)$ and $\partial^2_z F(x, 0)$ are known. In the proof, we construct a sequence of collisions approaching the initial timeline to simulate a boundary collision. This technique enables us to overcome the difficulties of this hyperbolic system without assumptions on the nonlinearity structure.
\end{abstract}

%%%%%%%%%%%%%%%%%%%%%%%%%%%%%%%%%%%%%%%%%%%%%%%%%%%%%%%%%%%%%%%%%%%%%%%
\section{Introduction}
% The Dirac equation, introduced by Paul Dirac in \cite{Dir}, is fundamental in quantum mechanics and quantum field theory for describing fermions. The nonlinear Dirac equation, which naturally arises as a simplification of the Dirac-Maxwell system (see for example \cite{Gro}), models interacting fermions across various physical contexts and has consequently been extensively studied. One notable example is the cubic nonlinearity $\left<\beta u, u\right>u$, which models self-interacting fermions in certain quantum field settings (see for example \cite{HehVonKerNes}).

The Dirac equation, introduced by Paul Dirac in \cite{Dir}, is fundamental in quantum mechanics and quantum field theory for describing fermions. Its nonlinear variants, which can be seen either as self‑interacting spinors (e.g. the cubic term $\langle\beta u, u\rangle u$) or as reductions of the coupled Dirac-Maxwell (and ultimately Einstein-Cartan-Sciama–Kibble) systems, have attracted extensive mathematical study. In particular, extensive symmetry classifications and explicit solutions were developed by Fushchich and Zhdanov \cite{FZ89}, while Esteban and S\'er\'e \cite{ES95} proved the existence of stationary (time-harmonic) states via variational methods. A recent centenary review of Dirac‑equation derivations appears in \cite{Sim25}, and the cubic self‑interaction induced by spacetime torsion in Einstein-Cartan-Sciama-Kibble gravity is discussed in \cite{HVDN76}. Finally, the full Maxwell-Dirac coupling has been shown to be well-posed and structurally rich in the works of Chadam \cite{Cha73}, Radford \cite{Rad03}, Glassey and Strauss \cite{GS79}, and Sparber and Markowich \cite{SM03}.

In this paper, we consider an inverse problem for 4 by 4 nonlinear Dirac equations of the following form
\begin{equation}\label{nonlinear Dirac equation}
    \begin{cases}
        i \partial_t u + i\alpha \cdot \nabla u - \beta u = F(x, u) \text{ on } (0, T) \times \mathbb{R}^3\\
        u(0, x) = \phi(x)
    \end{cases}
\end{equation}
where the notations mean the following:
\begin{enumerate}
    \item $u: [0, T]\times \mathbb{R}^3 \xrightarrow{} \mathbb{C}^4$ is the solution of \eqref{nonlinear Dirac equation}, with the variables denoted by $(t, x)$;
    \item For $\xi \in \mathbb{R}^3$, $\alpha \cdot \xi = \sum_{j=1}^{3} \xi_j\alpha_j$, and $\alpha_j$ are $4 \times 4$ Pauli matrices defined by
    \begin{equation*}
        \alpha_j = \begin{pmatrix}
            0& \sigma_j \\ \sigma_j& 0
        \end{pmatrix},
        \sigma_1 = \begin{pmatrix}
            0& 1 \\ 1& 0
        \end{pmatrix},
        \sigma_2 = \begin{pmatrix}
            0& -i \\ i& 0
        \end{pmatrix},
        \sigma_3 = \begin{pmatrix}
            1& 0 \\ 0& -1
        \end{pmatrix};
    \end{equation*}
    \item $\beta$ is the $4 \times 4$ matrix
    \begin{equation*}
        \beta = \begin{pmatrix}
            I_2 & 0 \\ 0 & -I_2
        \end{pmatrix}
    \end{equation*}
    and $I_2$ is the $2 \times 2$ identity matrix;
    \item $F: \mathbb{R}^3 \times \mathbb{C}^4 \xrightarrow{} \mathbb{C}^4$ is a smooth nonlinearity with $F(x, 0) = 0$, and its partial derivatives of all orders are bounded on sets of the form $\mathbb{R}^3 \times K$, where $K \subset \mathbb{C}^4$ compact. We denote the variables as $(x, z)$.
\end{enumerate}

We now phrase the main question of our paper. Assume that given any initial data $u(0) = \phi$, we are able to read the solution $u(T)$ at time $T$; in other words, we have complete knowledge of the map $u(0)\mapsto u(T)$. Does this uniquely determine the nonlinearity $F$? We will give an affirmative answer for data in the appropriate space of functions to be specified below.

To precisely state our main result, we first invoke a local existence result from \cite{Rau}, adapted to our setting.
\begin{theorem}\label{local existence theorem}
    (\cite{Rau} Theorem 6.3.1) If $s > 3/2$, then there is a $T>0$ and a unique solution $u \in C([0, T]; H^s(\mathbb{R}^3))$ to the semilinear initial value problem defined by the partial differential equation \eqref{nonlinear Dirac equation} together with the initial condition
    \begin{equation*}
    u(0, x) = \phi(x) \in H^s(\mathbb{R}^3).
    \end{equation*}
    The time $T$ can be chosen uniformly for $\phi$ from bounded subset of $H^s(\mathbb{R}^d)$. Consequently, there is a $T^* \in ]0, \infty]$ and a maximal solution $u \in C([0, T^*[; H^s(\mathbb{R}^d))$. If $T^* < \infty$, then
    \begin{equation*}
    \lim_{t \to T^*} \| u(t) \|_{H^s(\mathbb{R}^d)} = \infty.
    \end{equation*}
\end{theorem}

The above theorem allows us to define bounded time solution map for sufficiently small initial data.

\begin{definition}
    Given $s > 3/2$ and $\delta>0$, denote
    \begin{equation*}
        H^s_{\delta}: = \{\phi \in H^s(\mathbb{R}^3): ||\phi||_{H^s} < \delta\}.
    \end{equation*}
    By Theorem \ref{local existence theorem} there exists some $T>0$ such that for any initial data $\phi \in H^s_{\delta}$, \eqref{nonlinear Dirac equation} has a unique solution $u$ up to time $T$. Then the time $T$ solution map is defined on $H^s_{\delta}$ by
    \begin{equation*}
        S(\phi) = u(T).
    \end{equation*}
\end{definition}

That is, given any energy-wise sufficiently small initial data, the time $T$ solution map returns the solution at time $T$. We can now state the main result.
\begin{theorem}\label{main theorem}
    Given smooth nonlinearity $F_1$ and $F_2$, suppose for some $T>0$ they each have a well-defined time $T$ solution map $S_1$, $S_2$ on the same domain $H^s_{\delta}$. Then $S_1 = S_2$ implies $\partial^3_z F_1(x, z) \equiv \partial^3_z F_2(x, z)$ for all $(x, z) \in \mathbb{R}^3 \times B_M$ where $B_M \subset \mathbb{C}^4$ is the closed ball centered at 0 with radius $M$, and
    \begin{equation*}
        M = \sup \{||\phi||_{L^{\infty}}: \phi \in H^s_{\delta} \cap C^{\infty}(\mathbb{R}^3)\}.
    \end{equation*}
    In addition, if for all $x \in \mathbb{R}^3$,
    \begin{equation*}
        \partial_z F_1(x, 0) \equiv \partial_z F_2(x, 0), \quad \partial^2_zF_1(x, 0) \equiv \partial^2_z F_2(x, 0),
    \end{equation*}
    then $F_1(x, z) = F_2(x, z)$ for all $(x, z) \in \mathbb{R}^3 \times B_M$.
\end{theorem}

Although our result is proved for smooth small-data solutions, it applies directly to several physical models governed by semilinear Dirac equations. In Einstein-Cartan-Sciama-Kibble gravity, coupling the Dirac field to torsion produces exactly the cubic, four-fermion self-interaction studied here. In condensed matter, “Dirac materials” such as graphene or topological insulators, effective Dirac equations with nonlinear corrections capture electron-electron and electron-phonon interactions, so one could, at least in principle, use bounded time scattering data to recover the interaction law. Similar semilinear Dirac systems also arise in nonlinear optics as envelope equations for spinor-like wavepackets in photonic lattices.

For simplicity of the notations, in the remaining part of the paper we shall only use $F$ instead of $F_1$ and $F_2$, and prove that $\partial_z^3 F$ can be determined from the solution map. Specifically, whenever we say $A$ determines $B$, we mean the following: for $i=1,2$, let $A_i, B_i$ be some quantities given by nonlinearity $F_i$, then $A_1 = A_2$ implies $B_1 = B_2$.

Below we make several remarks about our result and its consequences.

\begin{remark}
    \hspace*{\parindent}
    \begin{enumerate}
        \item Note that if $\partial^3_z F(x, z)$ has been determined on $\mathbb{R}^3 \times B_M$, then for any $x \in \mathbb{R}^3$, $F(x, z)$ is determined up to a degree 2 polynomial in $z$ whose constant term is 0. With the extra assumption on the first and second derivatives, one can then determine the nonlinearity. We provide an example to illustrate the theorem in terms of the most important nonlinearity for the Dirac equation, namely the cubic nonlinearity. Suppose there exists some smooth nonlinearity $F$ with vanishing first and second derivative at $z = 0$, and whose time $T$ solution map coincides with that of the cubic nonlinearity. Then, $F$ must match with the cubic nonlinearity in a neighborhood of 0.
        \item If one can take $s < 3/2$ by imposing additional assumptions on $F$ such that the bounded solution map is well-defined for small data (for instance, see \cite{EscVeg} and \cite{MacNakOza}), then there is no uniform bound on the $L^{\infty}$ norm in $H^s_{\delta} \cap C^{\infty}(\mathbb{R}^3)$. In this case, $M = + \infty$, that is $F$ can be determined on the entire $\mathbb{R}^3 \times \mathbb{C}^4$. (For example, for $\phi \in C^{\infty}_c(\mathbb{R}^3)$, we can consider $\phi_{\epsilon}(x) = \epsilon^{-k}\phi(x/\epsilon)$ for $k+s < n/2$, then $||\phi_{\epsilon}||_{H^s} \xrightarrow{} 0$ while $||\phi_{\epsilon}||_{L^{\infty}} \xrightarrow{} \infty$). 
        \item As a direct result of Theorem \ref{main theorem}, if we have additional assumption that both nonlinearities are analytic in $z$, then $F_1 = F_2$ on $\mathbb{R}^3 \times \mathbb{C}^4$.
    \end{enumerate}
\end{remark}

Now we explain the main ingredients and challenges in our paper. The method used in this paper involves perturbative techniques around smooth solutions, adapted from the framework of higher order linearization and microlocal analysis established in \cite{KurLasUhl} by Kurylev, Lassas and Uhlmann. This approach leverages nonlinear wave interactions to recover information about the derivatives of nonlinearity through carefully constructed wave collisions. After \cite{KurLasUhl, KurLasUhl2}, there has been significant progress on inverse problem for nonlinear hyperbolic equations, see for example \cite{UhlZha, HinUhlZha, BalKujLasLii, CheLasOksPat, UhlWan, LasUhlWan, Tzo, KalRun, UhlZha2, KurLasOksUhl, WanZho, CheLasOksPat2, LasUhlWan2, UhlZha3, BarSte, SabSte, FeiOks, HinUhlZha2}, and for an overview of the recent progress in \cite{GunJia, Las}.

Our work presents additional challenges compared to previous works which stem from the following two features of the problems: 1) Dirac equation is a first order PDE system rather than a scalar equation; 2) we do not assume any specific structural condition for the nonlinearity. More specifically, we encounter the following difficulties:

\begin{itemize}
    \item \textbf{Restrictions on initial data:}  Although the Dirac equation remains of real principal type (see Section \ref{section: microlocal analysis} for details), it differs from the scalar case (see for example \cite{KurLasUhl, LasUhlWan2, HinUhlZha, UhlZha, SabUhlWan}) in that its principal part is a non-diagonal matrix that is singular on the characteristic set. As we will show, the principal symbol of the higher order linearization must satisfy a matrix-coefficient transport equation along bicharacteristics, with initial values restricted to specific subspaces of $\mathbb{C}^4$ based on each bicharacteristic. We emphasize that this phenomenon also does not happen to all systems, it arises when the principal symbol of the operator is not the zero matrix on the characteristic set (see Section \ref{section: distorted plane wave and first order linearization}). For instance, this does not happen to system of wave equation, as the principal symbol is simply the principal symbol of wave operator times identity matrix, which vanishes on the characteristic set.
    
    \item \textbf{Non-trivial propagation:} Another problem is caused by the fact that Dirac operator is of degree 1. For comparison we consider a semilinear wave equation $\Box u = f(u)$ with unknown nonlinearity $f$. The terms in higher order linearization satisfy equations of the form
    \begin{equation*}
        (\Box - f'(u))w = g.
    \end{equation*}
    In this context, multiplication by $f'(u)$ is 2 degree lower than $\Box$, which, according to symbol calculus, does not appear in the transport equation satisfied by the principal symbol of $w$. As a result, $w$ has constant principal symbol along bicharacteristics. For the Dirac equation, however, multiplication by $\partial_z F$ is on the subprincipal level because Dirac operator is a degree 1 operator. Consequently $F$ is always involved in the transport equation (see Section \ref{subsection: causal inverse}), leading to non-trivial propagation along bicharacteristic. Moreover, being a system again makes things more complicated as a first order linear system of ODE whose matrix coefficient depends on the variable usually does not have an explicit solution.
    
    \item \textbf{No structure assumption on $F$:} There have been some works on systems of PDE and they also encountered the problem of non-trivial propagation. In \cite{CheLasOksPat} and \cite{CheLasOksPat2}, the problem was reduced to collecting information about the non-Abelian broken ray transform, followed by inversion of a novel non-Abelian broken light ray transform. However, this approach relies on the specific structure of nonlinearity. For general smooth nonlinearity $F$, when performing 3 wave interaction and let the three incoming rays converge to the same one, the information at the collision point will be of the form $\partial^3_zF(x, u, Bv, Bv, Bv)$ for some parallel transport operator $B$ (see \cite{CheLasOksPat}) and vector $v$. In general it is hard to convert it into usable information, which is of the form $BG$ where $G$ doesn't depend on nonlinearity. It is also worth mentioning that another direct consequence of not having structural assumption is that perturbation around zero solution does not recover information of the nonlinearity away from 0, and perturbation around non-trivial solutions further complicates the aforementioned transport equation.
\end{itemize}

The main strategy for overcoming these challenges is to construct a sequence of collisions that approaches the initial time $t = 0$. Let us outline the general idea below. We fix a bicharacteristic from time 0 to time $T$, and again use higher order linearization and microlocal analysis to observe how information propagates along this bicharacteristic line. For suitable initial value at time $0$, its final value at time $T$ would be determined by the solution map. If, on the other hand, the starting position of the information is at a point on the bicharacteristic that is very close to the actual beginning point at $t = 0$, we can expect its outcome at time $T$ to be very close to the outcome that would result from initiating the information at the actual starting point. This observation suggests the main idea of the proof: we perform 3-wave interactions multiple times along a fixed bicharacteristic, allowing the collision point to approach the boundary at $t = 0$. The limiting case would correspond to the propagation of collision data from $t = 0$ to $t = T$. With some linear algebra the third derivative of $F$ at the boundary can be fully determined.

\vspace{.2in}

The organization of the paper is as follows. In Section \ref{section: asymptotic analysis} we perform asymptotic analysis to compute the equations satisfied by each linearization term. In Section \ref{section: microlocal analysis} we present some results about distribution multiplication and propagation of singularity for future use. In Section \ref{section: distorted plane wave and first order linearization} we define and prove the bijectivity of initial data to final data map. In Section \ref{section: third order linearization}, we perform a 3-wave interaction and compute the equation satisfied by the principal symbol of the third order linearization term. In Section \ref{section: limit of collisions to boundary} we obtain a one parameter family of collisions and compute the limiting case. In Section \ref{section: proof of main theorem} we combine the previous results and finish the proof of Theorem \ref{main theorem} with some linear algebra arguments.

%%%%%%%%%%%%%%%%%%%%%%%%%%%%%%%%%%%%%%%%%%%%%%%%%%%%%%%%%%%%%%%%%%%%%%%
\section*{Acknowledgement}

The author would like to thank Gunther Uhlmann for suggesting this problem and for many helpful discussions. The author would also like to thank Gunther Uhlmann and Ioan Bejenaru for all the helpful discussions and suggestions.

%%%%%%%%%%%%%%%%%%%%%%%%%%%%%%%%%%%%%%%%%%%%%%%%%%%%%%%%%%%%%%%%%%%%%%%
\section{Asymptotic analysis}\label{section: asymptotic analysis}
To simplify the notation, starting now we will perform the computations for $F(z)$ instead of $F(x, z)$, thus all derivatives below are taken with respect to $z$. As one will see the computations up to the end of Section \ref{section: limit of collisions to boundary} are exactly the same when $F$ also depends on $x$. We will come back to $F(x, z)$ in Section \ref{section: proof of main theorem} when we prove Theorem \ref{main theorem}.

We now introduce perturbations around small smooth initial data. Consider initial data of the form $\phi_{\epsilon} = \phi + \sum_1^3 \epsilon_j\phi_j$ where $\phi \in H^s_{\delta}$ is smooth. Let $u, \ue$ denote the solution to \eqref{nonlinear Dirac equation} with initial data $\phi, \phi_{\epsilon}$, respectively. By Theorem \ref{local existence theorem}, $u$ is smooth in $x$ for $t \in [0, T]$ and continuous in $t$. Since $u$ satisfies \eqref{nonlinear Dirac equation}, a standard bootstrap argument shows that $u$ is smooth in both $t$ and $x$.

Let $w_j$ be the solution of the following linear Dirac equation
\begin{equation*}
    \begin{cases}
        i \partial_t w_j + i\alpha \cdot \nabla w_j - \beta w_j - F'(u)w_j = 0 \text{ on } (0, T) \times \mathbb{R}^3,\\
        w_j(0) = \phi_j.
    \end{cases}
\end{equation*}
If $w$ solves the following non-homogeneous linear Dirac equation
\begin{equation}\label{modified linear Dirac equation}
    \begin{cases}
        i \partial_t w + i\alpha \cdot \nabla w - \beta w - F'(u)w = f \text{ on } (0, T) \times \mathbb{R}^3,\\
        w(0) = 0,
    \end{cases}
\end{equation}
we denote $Q(f) = w$, note that $Q$ would depend on $u$. Write the Taylor expansion of $F$ around $u$ as
\begin{align*}
    F(\ue) &= F(u) + F'(u)(\ue-u) + F^{(2)}(u, \ue-u, \ue-u)\\
    &+F^{(3)}(u, \ue-u, \ue-u, \ue-u) + O(|\ue-u|^4)
\end{align*}
where
\begin{equation*}
    F^{(k)}(u, v_1, ...,v_k) = \sum_{|\gamma|=k} \frac{1}{\gamma!} (\partial^{\gamma}F)(u)v_1^{\gamma_1}...v_k^{\gamma_k}.
\end{equation*}
Then $(\ue - u - \sum \epsilon_j w_j)(0) = 0$ and
\begin{align*}
    &(i \partial_t + i\alpha \cdot \nabla - \beta - F'(u))(\ue - u - \sum \epsilon_j w_j)\\
    &= F(\ue)-F(u)-F'(u)(\ue - u) \\
    &= F^{(2)}(u, \ue - u, \ue - u) + F^{(3)}(u, \ue - u, \ue - u, \ue - u) + O(|\ue - u|^4).
\end{align*}
That is,
\begin{equation*}
    \begin{split}
        \ue - u &= \sum \epsilon_j w_j + Q\left( F^{(2)}(u, \ue - u, \ue - u) \right. \\
        &\left. + F^{(3)}(u, \ue - u, \ue - u, \ue - u) + O(|\ue - u|^4) \right) .
    \end{split}
\end{equation*}
By repeatedly substituting $\ue - u$ with the right hand side and use the linearity of $Q$ and symmetry of $F^{(2)}, F^{(3)}$, we have 
\begin{equation}\label{eq: expansion of u_epsilon}
    \ue = u + \sum \epsilon_j w_j + \sum_{i \leq j} \epsilon_i\epsilon_jw_{ij} + \sum_{i \leq j \leq k} \epsilon_i\epsilon_j\epsilon_kw_{ijk} + O(|\epsilon|^4)
\end{equation}
where for $i < j < k$,
\begin{align}\label{eq: linearization terms of u_epsilon}
    w_{ij} = &Q(2F^{(2)}(u, w_i, w_j)),\nonumber\\
    w_{ii} = &Q(F^{(2)}(u, w_i, w_i)),\nonumber\\
    \begin{split}
        w_{ijk} = &Q\left(6F^{(3)}(u, w_i, w_j, w_k) + 2F^{(2)}(u, w_{ij}, w_k)\right. \\
        &\left. + 2F^{(2)}(u, w_{ik}, w_j) + 2F^{(2)}(u, w_{jk}, w_i)\right),
    \end{split}\nonumber\\
    \begin{split}
        w_{iij} = &Q\left( 3F^{(3)}(u, w_i, w_i, w_j) + 2F^{(2)}(u, w_{ii}, w_j) \right. \\
        &\left. + 2F^{(2)}(u, w_{ij}, w_i)\right),
    \end{split}\nonumber\\
    w_{iii} =& Q\left(F^{(3)}(u, w_i, w_i, w_i) + 2F^{(2)}(u, w_{ii}, w_i)\right).
\end{align}

The next perturbation result from \cite{Rau}, again adapted to current scenario, shows that the above terms exist and the remaining term in \eqref{eq: expansion of u_epsilon} is of $|\epsilon|^4$ order in $C([0, T]; H^s(\mathbb{R}^3))$.

\begin{theorem}\label{thm: smooth depandence on initial data}
    (\cite{Rau} Theorem 6.5.2) If $u \in C([0, T]; H^s(\mathbb{R}^3))$ is a solution of \eqref{nonlinear Dirac equation}, then the map $\psi \mapsto w$ from initial data to solution for \eqref{nonlinear Dirac equation} is smooth from a neighborhood of $u(0)$ in $H^s(\mathbb{R}^3)$ to $C([0, T]; H^s(\mathbb{R}^3))$, and the first, second and third derivatives of $\phi_{\epsilon} \mapsto u_{\epsilon}$ with respect to $\epsilon= (\epsilon_1, \epsilon_2, \epsilon_3)$ at $\epsilon = 0$ are given by $w_i$, $w_{ij}$ and $w_{ijk}$ satisfying the above equations. Derivatives of each order are uniformly bounded on the neighborhood.
\end{theorem}
By \eqref{eq: expansion of u_epsilon}, the solution map satisfies 
\begin{align*}
    S(\phi_{\epsilon}) &= \ue(T)\\
    &=u(T) + \sum \epsilon_j w_j(T) + \sum_{i \leq j} \epsilon_i\epsilon_jw_{ij}(T)\\
    &+ \sum_{i \leq j \leq k} \epsilon_i\epsilon_j\epsilon_kw_{ijk}(T) + O(|\epsilon|^4).
\end{align*}
By Theorem \ref{thm: smooth depandence on initial data}, the solution map $S: H^s_{\delta} \to H^s(\mathbb{R}^3)$ is smooth on a neighborhood of $\phi$, we denote its Fr\'echet derivatives as
\begin{equation*}
    \partial_{\epsilon^{\gamma}}(S(\phi_{\epsilon}))|_{\epsilon = 0}
\end{equation*}
for multi-index $\gamma$. For example, $\phi_j \mapsto \partial_{\epsilon_j}(S(\phi+\epsilon_j\phi_j))|_{\epsilon_j = 0}$ is a bounded linear operator satisfying
\begin{equation*}
    \lim_{\epsilon_j \to 0}\frac{||S(\phi+\epsilon_j\phi_j) - S(\phi)-\partial_{\epsilon_j}S(\phi + \epsilon_j\phi_j)|_{\epsilon_j=0}||_{H^s(\mathbb{R}^3)}}{\epsilon_j||\phi_j||_{H^s(\mathbb{R}^3)}} = 0.
\end{equation*}
Same for the higher order derivatives. Then by Theorem \ref{thm: smooth depandence on initial data}, we have the following proposition.

\begin{proposition}\label{solution map prop}
    Let the initial data be of the form $\phi_{\epsilon} = \phi + \sum_{j=1}^3 \epsilon_j\phi_j$ for $\phi \in H^s_{\delta}$, $\phi_j \in H^s(\mathbb{R}^3)$. Then $w_j, w_{ij}$ and $w_{ijk}$ defined in \eqref{eq: linearization terms of u_epsilon} are well-defined in $H^s(\mathbb{R}^3)$. The first, second and third derivatives of the solution map give the solution maps of the corresponding linearization terms:
    \begin{align*}
        &\partial_{\epsilon_j}(S(\phi_{\epsilon}))|_{\epsilon=0} = w_j(T),\\
        &\partial_{\epsilon_i,\epsilon_j}(S(\phi_{\epsilon}))|_{\epsilon=0} = w_{ij}(T),\\
        &\partial_{\epsilon_i,\epsilon_j,\epsilon_k}(S(\phi_{\epsilon}))|_{\epsilon=0} = w_{ijk}(T).
    \end{align*}
\end{proposition}

%%%%%%%%%%%%%%%%%%%%%%%%%%%%%%%%%%%%%%%%%%%%%%%%%%%%%%%%%%%%%%%%%%%%%%%
\section{Microlocal analysis}\label{section: microlocal analysis}
In this section we mainly study the operator
\begin{equation*}
    P = i\partial_t + i\alpha \cdot \nabla - \beta - F'(u).
\end{equation*}
The definition of real principal type operator for systems of pseudodifferential operators is given in \cite{Den}.
\begin{definition}
    (\cite{Den} Definition 3.1) An $N\times N$ system $P$ of pseudodifferential operators on $X$ with principal symbol $p(x, \xi)$ is of real principal type at $(y, \eta)\in T^*X\backslash 0$ if there exists an $N\times N$ symbol $\Tilde{p}(x, \xi)$ such that
    \begin{equation*}
        \Tilde{p}(x,\xi)p(x,\xi) = q(x,\xi)\cdot Id_N
    \end{equation*}
    in a neighborhood of $(y, \eta)$, where $q(x, \xi)$ is a scalar symbol of real principal type. We say that $P$ is of real principal type in $T^*X$ if it is at every point.
\end{definition}
The matrices satisfies the following relations:
\begin{equation*}
    \alpha_j\alpha_k + \alpha_k\alpha_j = 2 \delta_{jk}I_4,\quad \alpha_j\beta = - \beta\alpha_j, \quad \beta^2 = I_4.
\end{equation*}
Hence $P$ is of real principal type with $\Tilde{p} = \tau - \alpha \cdot \xi$ and $q = \tau^2 - |\xi|^2$. $\Tilde{P}= -i\partial_t + i\alpha \cdot \nabla$ is a quantization of $\Tilde{p}$, and
\begin{equation*}
    \Tilde{P}P = \Box - (-i\partial_t + i\alpha\cdot\nabla) (\beta + F'(u))
\end{equation*}
whose principal part is the wave operator. When we say the Hamiltonian vector field related to $p$ or $P$ we mean the Hamiltonian vector field $H_q$ generated by the scalar symbol $q$.

%%%%%%%%%%%%%%%%%%%%%%%%%%%%%%%%%%%%%%%%%%%%%%%%%%%%%%%%%%%%%%%%%%%%%%%
\begin{subsection}{Notations}\label{subsection: notations}
We first introduce some notations.
\begin{align*}
    \text{Char}(P) &= \{ (t, x, \tau, \xi): \det(-\tau-\alpha\cdot\xi) = 0, \xi \neq 0 \} \\
    &= \{ (t, x, \tau, \xi): (\tau^2 - |\xi|^2)^2 = 0, \xi \neq 0 \}\\
    &= L\mathbb{R}^4 \backslash 0.
\end{align*}
the set of lightlike covectors. For some $T' < 0$, let $y^{T'}_j = (T', x^{T'}_j)$ and $\eta_j$ a light-like covector. $\gamma_{y, \eta}$ is the geodesic passing through $y$ in the direction of $\eta^{\sharp}$. Denote
\begin{align*}
    &\mathcal{V}(y^{T'}_j, \eta_j, \epsilon) = \{\eta \in T^*_{y^{T'}_j}\mathbb{R}^{n+1}: |\eta - \eta_j|<\epsilon, |\eta| = |\eta_j|\},\\
    &K(y^{T'}_j, \eta_j, \epsilon) = \{\gamma_{y^{T'}_j, \eta}(s): \eta \in L^+_{y^{T'}_j}\mathbb{R}^4 \cap \mathcal{V}(y^{T'}_j, \eta_j, \epsilon), s\in (0, \infty)
    \},\\
    &\Sigma(y^{T'}_j, \eta_j, \epsilon) = \{(y^{T'}_j, r\eta^{\flat}) \in T^*\mathbb{R}^{n+1}: \eta \in \mathcal{V}(y^{T'}_j, \eta_j, \epsilon), r \neq 0\}.
\end{align*}
Let $\Lambda(y^{T'}_j, \eta_j, \epsilon)$ be the Lagrangian manifold that is the flowout of lightlike covectors in $\Sigma(y^{T'}_j, \eta_j, \epsilon)$ via $H_q$ in the future direction. We shall use $\mathcal{V}_j, K_j, \Sigma_j, \Lambda_j$ for simplicity, and denote
\begin{equation}\label{bicharacteristic}
    \Gamma_j(s) = (y_j, \eta_j) + 2s(\tau_j, -\xi_j, 0,0) 
\end{equation}
the bicharacteristic where $\Gamma_j(0) = (y_j, \eta_j) = (0, x_j, \tau_j, \xi_j)$, $\Gamma_j(\frac{T'}{2\tau_j}) = (y_j^{T'}, \eta_j)$. The end point is $\Gamma_j(\frac{T}{2\tau_j}) = (y_j^T, \eta_j) = (T, x_j^T, \tau_j, \xi_j)$.

To perform multiple wave interactions, we introduce the following notations
\begin{align*}
    &K_{ij} = K_i \cap K_j, \quad K_{123} = K_1 \cap K_2 \cap K_3,\\
    &\Lambda_{ij} = N^*K_{ij}, \quad \Lambda_{123} = N^*K_{123}.
\end{align*}
If $K_1, K_2, K_3$ intersect transversally, then conormal fiber of $K_{123}$ is spanned by the conormal fiber of each $K_j$ so would be a timelike subspace, hence $K_{123}$ is a spacelike 1 dimensional submanifold. Let $q$ be a point in $K_{123}$ and $(q, \zeta)$ a lightlike direction in $\Lambda_{123}$, then the corresponding bicharacteristic only intersect $\Lambda_{123}$ once. As $\epsilon \mapsto 0$, we see that $K_j$ goes to $\Gamma_j$ and $K_{123}$ converges to a point.
\end{subsection}

%%%%%%%%%%%%%%%%%%%%%%%%%%%%%%%%%%%%%%%%%%%%%%%%%%%%%%%%%%%%%%%%%%%%%%%
\begin{subsection}{Lagrangian Distributions and Intersecting Lagrangians}\label{subsection: lagrangian distributions and intersecting lagrangians}
Given a conic Lagrangian submanifold $\Lambda$ of $T^*\mathbb{R}^4 \backslash 0$, $\mathcal{I}^{\mu}(\Lambda; \mathbb{C}^4)$ denotes all the corresponding Lagrangian distributions associated with $\Lambda$ of order $\mu$ taking value in $\mathbb{C}^4$. The wavefront set of such distribution is referred to the union of the wavefront set of each coordinate, and the wavefront set of any such distribution would be in $\Lambda$. The principal symbol of a Lagrangian distributions is invariantly defined on the cotangent bundle and takes value in $\mathbb{C}^4 \otimes \Omega^{1/2} \otimes \mathcal{L}$ where $\Omega^{1/2}$ is the half density bundle and $\mathcal{L}$ is the Maslov bundle. The principal symbol of $w \in \mathcal{I}^{\mu}(\Lambda; \mathbb{C}^4)$ is usually denoted by $\sigma_{\mu}(w)$, and most of the time we shall omit the $\mu$ when the level of the principal symbol of clear. Furthermore, when determining distorted plane wave, we shall let it be a classical conormal distribution $\mathcal{I}^{\mu}_{cl}(\Lambda; \mathbb{C}^4)$, meaning the principal symbol is homogeneous.

Let $\Lambda_0, \Lambda_1$ be two cleanly intersecting conic Lagrangian submanifolds, then $\mathcal{I}^{p,l}(\Lambda_0, \Lambda_1; \mathbb{C}^4)$ is the set of all paired Lagrangian distributions associated with $(\Lambda_0, \Lambda_1)$ taking value in $\mathbb{C}^4$. If $w \in \mathcal{I}^{p,l}(\Lambda_0, \Lambda_1; \mathbb{C}^4)$, then $WF(w) \in \Lambda_0 \cup \Lambda_1$, and if $A \in \Psi^0(\Lambda_0; M_{4\times 4}(\mathbb{C}))$ pseudodifferential operator of order 0 on $\Lambda_0$ taking value in $4 \times 4$ complex matrices, and $WF(A) \cap \Lambda_1 = \emptyset$, then $Aw \in \mathcal{I}^{p+l}(\Lambda_0; \mathbb{C}^4)$. Similarly if $B \in \Psi^0(\Lambda_1; M_{4\times 4}(\mathbb{C}))$ and $WF(B) \cap \Lambda_0 = \emptyset$, then $Bw \in \mathcal{I}^{p}(\Lambda_1; \mathbb{C}^4)$. The principal symbol of $w$ on $\Lambda_0$ and $\Lambda_1$ can thus be defined, and they satisfy a compatibility condition that depends only on the geometry of the Lagrangian submanifolds.

We shall omit the $\mathbb{C}^4$ in $\mathcal{I}^{\mu}(\Lambda; \mathbb{C}^4)$ and $\mathcal{I}^{p,l}(\Lambda_0, \Lambda_1; \mathbb{C}^4)$ for simplicity. The specific definition of Lagrangian distributions can be found in chapter 25 of \cite{Jac}, as well as \cite{Hor}. And for paired Lagrangian distributions we refer to \cite{MelUhl}, \cite{GreUhl} and \cite{GuiUhl}.

Finally, we also include some distribution multiplication results here for future use.
\begin{lemma}\label{multiplication 1}
    (\cite{LasUhlWan2} Lemma 3.3) Let $u \in \mathcal{I}^{\mu}(\Lambda_1)$, $v \in \mathcal{I}^{\mu'}(\Lambda_2)$. Then we can write $w = uv$ as $w = w_1 + w_2$ where
    \begin{equation*}
        \begin{split}
            &w_1 \in \mathcal{I}^{\mu, \mu'+1}(\Lambda_{12}, \Lambda_1), \quad WF(w_1) \cap \Lambda_2 = \emptyset,\\
            &w_2 \in \mathcal{I}^{\mu', \mu+1}(\Lambda_{12}, \Lambda_2), \quad WF(w_2) \cap \Lambda_1 = \emptyset.
        \end{split}
    \end{equation*}
    Moreover, for any $(q, \zeta) \in \Lambda_{12}\backslash(\Lambda_1\cup\Lambda_2)$, we can write $\zeta = \zeta_1 + \zeta_2$ in a unique way such that $\zeta_i \in N^*_qK_i, i = 1, 2$. Microlocally away from $\Lambda_1 \cup \Lambda_2$, $uv \in \mathcal{I}^{\mu+\mu'+1}(\Lambda_{12})$ and the principal symbol of $uv$ satisfies
    \begin{equation*}
        \sigma_{\Lambda_{12}}(uv)(q, \zeta) = (2\pi)^{-1}\sigma_{\Lambda_1}(u)(q, \zeta_1) \cdot \sigma_{\Lambda_2}(v)(q, \zeta_2).
    \end{equation*}
\end{lemma}

\begin{lemma} \label{multiplication 2}
    (\cite{LasUhlWan2} Lemma 3.6) Assume that $u \in \mathcal{I}^{\mu}(\Lambda_3)$, $v \in \mathcal{I}^{p, l}(\Lambda_{12}, \Lambda_1)$ are compactly supported near $K_{123}$. For $\epsilon > 0$ sufficiently small, we can write $w = uv$ as $w = w_0 + w_1 + w_2$ where
    \begin{equation*}
        \begin{split}
            &w_0 \in \mathcal{I}^{\mu+p+l-1/2}(\Lambda_{123}), \quad w_1 \in \mathcal{D}'(\mathbb{R}^4; \cup_{i=1}^3\Lambda_i(\epsilon)),\\
            &w_2 \in \mathcal{I}^{p, \mu+1}(\Lambda_{13}, \Lambda_1) + \mathcal{I}^{\mu, p+1}(\Lambda_{13}, \Lambda_3) + \mathcal{I}^{p,l}(\Lambda_{12}, \Lambda_1).
        \end{split}
    \end{equation*}
    Here $\Lambda^{(1)}(\epsilon)$ is a conic $\epsilon$-neighborhood of $\cup_{i=1}^3\Lambda_i$. Moreover, for $q \in K_{123}$ and $\zeta \in N^*_qK_{123}\backslash(\cup_{i=1}^3 \Lambda_i)$, we can write $\zeta = \zeta_1 + \zeta_2 + \zeta_3$ uniquely for $\zeta_i \in N^*_qK_i$. The principal symbol of $w_0$ satisfies
    \begin{equation*}
        \sigma_{\Lambda_{123}}(w_0)(q, \zeta) = (2\pi)^{-1}\sigma_{\Lambda_3}(u)(q, \zeta_3) \cdot \sigma_{\Lambda_{12}}(v)(q, \zeta_1+\zeta_2).
    \end{equation*}
\end{lemma}

\end{subsection}

%%%%%%%%%%%%%%%%%%%%%%%%%%%%%%%%%%%%%%%%%%%%%%%%%%%%%%%%%%%%%%%%%%%%%%%
\begin{subsection}{Causal inverse}\label{subsection: causal inverse}
We first invoke a simplified version of propagation of singularity theorem for systems that will be used later, specifically we left out the polarization part as it will not be used.
\begin{theorem}\label{propagation of singularity}
    (\cite{Den} Theorem 4.2) Let $P$ be an $N \times N$ system of pseudodifferential operators on a manifold $X$ and let $u \in \mathcal{D}'(X, \mathbb{C}^n)$. Assume that $P$ is of real principal type at $(y, \eta) \in \text{Char}(P)$ and that $y \notin WF(Pu)$. Then, over a neighborhood of $(y, \eta)$ in $\text{Char}(P)$, $WF(u)$ is a union of bicharacteristics of $P$.
\end{theorem}
% Suppose $u$ is the solution to $Pu = F(u)$ with smooth initial data $u(0) \in C^{\infty}(\mathbb{R}^3)$, then away from $WF(F(u))$, $WF(u)$ is a union of bicharacteristics. 
As for the symbol calculus, we have the following theorem.
\begin{theorem}\label{transport equation for principal symbol}
    (\cite{HanRoh} Theorem 3.1) Let $P$ be an $N \times N$ real principal type system of pseudodifferential operators of order $m$ on a manifold $X$. Assume $\Lambda$ a homogeneous Lagrangian submanifold of $T^*X\backslash0$ such that $\Lambda \in \text{Char}(P)$. If $A \in \mathcal{I}^{\mu}(X, \Lambda; \Omega_X^{1/2} \otimes \mathbb{C}^N)$ and $a \in S^{\mu+n/4}(X, \Lambda; \Omega_X^{1/2} \otimes \mathbb{C}^N)$ is a principal symbol of $A$ such that $pa = 0$, it follows that $B = PA \in \mathcal{I}^{m+\mu-1}(X, \Lambda; \Omega_X^{1/2} \otimes \mathbb{C}^N)$ has principal symbol $b \in S^{m+\mu+n/4-1}(X, \Lambda; \Omega_X^{1/2} \otimes \mathbb{C}^N)$ satisfying
    \begin{equation*}
        (\mathcal{L}_{H_q}+\frac{1}{2}\{\Tilde{p}, p\}+i\Tilde{p}p^s)a = i\Tilde{p}b.
    \end{equation*}
    Here $\mathcal{L}_{H_q}$ is the Lie derivative with respect to the Hamiltonian vector field $H_q$, $p$ and $\Tilde{p}$ are principal symbols of $P$ and $\Tilde{P}$ respectively, and $p^s$ is the subprincipal symbol of $P$.
\end{theorem}
In particular, since in our case $P$ is independent of $y$, the transport equation is
\begin{equation*}
    [\mathcal{L}_{H_q} - i(\tau-\alpha\cdot\xi)(\beta + F'(u))]a = i(\tau-\alpha\cdot\xi)b
\end{equation*}
where $q = \tau^2 - |\xi|^2$ and $H_q$ is the Hamiltonian vector field of the standard wave operator
\begin{equation*}
    H_q = 2\tau\frac{\partial}{\partial t} - 2\xi \cdot \frac{\partial}{\partial x}.
\end{equation*}

Now we compute the causal inverse of $P$, recall that $Q(f)$ is the solution to \eqref{modified linear Dirac equation}.

\begin{proposition}\label{causal inverse 1}
If $f \in \mathcal{I}^{\mu}(\Lambda)$, and $w = Q(f)$, then $w \in \mathcal{I}^{\mu-1/2,-1/2}(\Lambda, \Lambda^g)$ where $\Lambda^g$ is the future flowout of $\Lambda \cap L^+\mathbb{R}^4$. On $\Lambda \backslash \Lambda^g$,
\begin{equation*}
    \sigma_{\Lambda}(w) = p^{-1}\sigma(f).
\end{equation*}
On $\Lambda^g$ the principal symbol satisfies
\begin{equation*}
    \begin{cases}
        (\mathcal{L}_{H_q}+i\Tilde{p}p^s)\sigma_{\Lambda^g}(w) = 0 \text{ on } \Lambda^g \backslash \partial \Lambda^g \\
        \sigma_{\Lambda^g}(w) = C\Tilde{p}\sigma(f) \text{ on } \partial\Lambda^g
    \end{cases}
\end{equation*}
for some constant $C$ that only depends on the geometry of $\Lambda$ and $\Lambda^g$.
\end{proposition}
\begin{proof}
    Note that if $w = Q(f)$, then $\Box w - \Tilde{P}(\beta+F'(u))w = \Tilde{P}f$ and $w_t = f -\alpha\cdot\nabla w - i (\beta+F'(u))w = 0$ on $(L^+\text{supp}(f))^c$. By \cite{BarGinPfa} and \cite{MelUhl}, $\Box w - \Tilde{P}(\beta+F'(u))$ has a causal inverse $Q' \in \mathcal{I}^{-3/2, -1/2}(N^*\text{Diag}, \Lambda_g)$ where
    \begin{equation*}
        N^*\text{Diag} = \{(z, \zeta, z', \zeta') \in T^*\mathbb{R}^8\backslash 0: z = z', \zeta = \zeta'\}
    \end{equation*}
    and $\Lambda_g$ is the Lagrangian obtained by flowing out $N^*\text{Diag} \cap (L\mathbb{R}^4 \times T^*\mathbb{R}^4)$ under $H_q$ lifted from the left factor (see section 3.1 of \cite{LasUhlWan}). Thus $Q': \mathcal{I}^{\mu'}(\Lambda) \xrightarrow{} \mathcal{I}^{\mu'-3/2,-1/2}(\Lambda, \Lambda^g)$. Because $w = Q(f) = Q'(\Tilde{P}f)$, plug in $\mu' = \mu+1$ we have $w \in \mathcal{I}^{\mu-1/2,-1/2}(\Lambda, \Lambda^g)$.

    To compute the principal symbol, notice that $p$ is invertible on $\Lambda \backslash \partial \Lambda^g$, so on the principal level $p\sigma_{\Lambda}(w) = \sigma(f)$ implies the first equation. The transport equation in the second equation is by Theorem \ref{transport equation for principal symbol}. To compute the initial data, notice that on $\Lambda \backslash \partial\Lambda^g$,
    \begin{equation*}
        p^{-1}\sigma(f) = q^{-1}\Tilde{p}\sigma(f),
    \end{equation*}
    and as $q = 0$ parametrizes the light cone, by compatibility condition $\sigma_{\Lambda^g}(w) = C\Tilde{p}\sigma(f)$ with some nonzero $C$ that only depends on the geometry of the intersecting submanifolds. We can also obtain the initial value using the fact that $w = Q'(\Tilde{p}f)$ and computations from proof of Proposition 6.6 in \cite{MelUhl}.
\end{proof}

\begin{proposition}\label{causal inverse 2}
If $f \in \mathcal{I}^{p, l}(\Lambda_0, \Lambda_1)$, and $w = Q(f)$, then $w \in \mathcal{I}^{p,l-1}(\Lambda_0, \Lambda_1)$. On $\Lambda_0 \backslash \Lambda_1$,
\begin{equation*}
    \sigma_{\Lambda_0}(w) = p^{-1}\sigma_{\Lambda_0}(f).
\end{equation*}
\end{proposition}
\begin{proof}
    $Q' \in \mathcal{I}^{-3/2, -1/2}(N^*\text{Diag}, \Lambda_g)$ so by \cite{LasUhlWan2} Lemma 3.4 
    \begin{equation*}
        Q': \mathcal{I}^{p', l'}(\Lambda_0, \Lambda_1) \xrightarrow{} \mathcal{I}^{p'-1, l'-1}(\Lambda_0, \Lambda_1).
    \end{equation*}
    Then $w \in \mathcal{I}^{p,l-1}(\Lambda_0, \Lambda_1)$ follows from $w = Q(f) = Q'(\Tilde{P}f)$. Away from the characteristic set $p$ is invertible on $\Lambda_0 \backslash \Lambda_1$ so the equation follows.
\end{proof}
    
\end{subsection}

%%%%%%%%%%%%%%%%%%%%%%%%%%%%%%%%%%%%%%%%%%%%%%%%%%%%%%%%%%%%%%%%%%%%%%%
\section{Distorted plane wave and first order linearization}\label{section: distorted plane wave and first order linearization}
First consider $P' = i\partial_t + i\alpha\cdot\nabla - \beta$, let $f_j \in \mathcal{I}_{cl}^{\mu+1/2}(\Sigma_j)$ and $\Tilde{w}_j$ the solution of
\begin{equation*}
    \begin{cases}
        P'\Tilde{w}_j = f_j\\
        \Tilde{w}_j = 0 \text{ on } (L^+supp(f_j))^c.
    \end{cases}
\end{equation*}
Note that $P'$ can be thought of as a special case of $P$ where the $F'(u)$ term is 0 everywhere. Then by Proposition \ref{causal inverse 1}, $\Tilde{w}_j \in \mathcal{I}^{\mu,-1/2}(\Sigma_j, \Lambda_j)$. Specifically if we only consider time 0 to time $T$, then $\Tilde{w}_j \in \mathcal{I}^{\mu}(\Lambda_j)$ and the principal symbol satisfies
\begin{equation}\label{distorted wave symbol equation}
    \begin{cases}
        (\mathcal{L}_{H_q} - i(\tau-\alpha\cdot\xi)\beta)\sigma(\Tilde{w}_j) = 0 \text{ on } \Lambda_j \\
        \sigma(\Tilde{w}_j) = C\Tilde{p}\sigma(f_j) \text{ on } \Sigma_j \cap \Lambda_j,
    \end{cases}
\end{equation}
which doesn't contain any nonlinearity so is fully known here. Let $\phi_j(x) = \Tilde{w}_j(0, x)$, by making $\mu$ sufficiently negative, $\phi_j \in H^s(\mathbb{R}^3)$ and hence perturbation result holds. Let $w_j$ solve
\begin{equation*}
    \begin{cases}
        Pw_j = 0\\
        w_j(0) = \phi_j.
    \end{cases}
\end{equation*}
Note that $P(w_j - \Tilde{w}_j) = F'(u)\Tilde{w}_j$, and initial data vanishes at 0. By Theorem \ref{propagation of singularity}, and the fact that $WF(\Tilde{w}_j) \subset \Lambda_j$, we have $WF(w_j - \Tilde{w}_j) \subset \Lambda_j$, hence $WF(w_j) \subset \Lambda_j$.

Let $\mathcal{R}_0$ be the restriction operator at time 0, that is $\mathcal{R}_0(w(t, x)) = w(0, x)$. We know $WF(w_j), WF(\Tilde{w}_j) \in \Lambda_j$, and $\Lambda_j$ does not intersect the conormal bundle of the time slice $\mathbb{R}^3_0 = \{(0, x): x \in \mathbb{R}^3\}$, so $\mathcal{R}_0 \in \mathcal{I}^{1/4}(\mathbb{R}^3_0, \mathbb{R}^4; C_0)$ is a Fourier integral operator of order $1/4$ (see \cite{Dui} Chapter 5.1) where the corresponding canonical relation is given by
\begin{equation*}
    C_0 = \{(x, \xi; 0, x, \tau, \xi)\}.
\end{equation*}
(It is worth mentioning that $\mathcal{R}_0$ can act on $w_j$ as a FIO because we can extend $w_j$ backward in time such that it is a distribution on an some $(-\epsilon, \epsilon) \times \mathbb{R}^3$.) For any $(0, x, \tau, \xi) \in \Lambda_j$,
\begin{align*}
    &\sigma(\phi_j)(x, \xi) = \sigma(\mathcal{R}_0)(x, \xi; 0, x, \tau, \xi)\sigma(w_j)(0, x, \tau, \xi),\\
    &\sigma(\phi_j)(x, \xi) = \sigma(\mathcal{R}_0)(x, \xi; 0, x, \tau, \xi)\sigma(\Tilde{w}_j)(0, x, \tau, \xi).
\end{align*}
Since $\sigma(\mathcal{R}_0)(x, \xi; 0, x, \tau, \xi) \neq 0$ (see \cite{Dui} Chapter 5.1), we have $w_j \in \mathcal{I}^{\mu}(\Lambda_j)$ and $\sigma(w_j)(0, x, \tau, \xi) = \sigma(\Tilde{w}_j)(0, x, \tau, \xi)$.

Using Proposition \ref{causal inverse 1} and the fact that $P'$ is fully known, we can determine what $\sigma(\Tilde{w}_j)(0, x_j, \tau_j, \xi_j)$ is by constructing appropriate $f_j$. Specifically, we rewrite \eqref{distorted wave symbol equation} along $\Gamma_j$ as
\begin{equation}\label{distorted plane wave ode}
    \begin{cases}
        \frac{d}{ds} \sigma(\Tilde{w}_j)(\Gamma_j(s)) = A_0(\Gamma_j(s))\sigma(\Tilde{w}_j)(\Gamma_j(s))\\
        \sigma(\Tilde{w}_j)(\Gamma_j(\frac{T'}{2\tau_j})) = C\Tilde{p}\sigma(f_j)(\Gamma_j(\frac{T'}{2\tau_j}))
    \end{cases}
\end{equation}
an ODE along $\Gamma_j$ where
\begin{equation*}
    A_0(y, \eta) = i(\tau - \alpha \cdot \xi)\beta.
\end{equation*}
Note that here $A_0$ doesn't depend on $y$, hence it stays constant along $\Gamma_j(s)$, $A_0 = A_0(\Gamma_j(s)) = i(\tau_j - \alpha\cdot\xi_j)\beta$. Then we simply have $\sigma(\Tilde{w}_j)(\Gamma_j(s)) = \exp(sA_0) \sigma(\Tilde{w}_j)(\Gamma_j(s))$. However there is some restriction on the initial data in \eqref{distorted plane wave ode}, namely it can only be in the kernel of $\Tilde{p}(\Gamma_j) = -\tau_j - \alpha\cdot\xi_j$. Moreover, we see that the image of $A_0(\Gamma_j(s))$ is precisely in $\text{ker}(-\tau_j - \alpha\cdot\xi_j)$, meaning $\sigma(\Tilde{w}_j)(\Gamma_j(s)) \in \text{ker}(-\tau_j - \alpha\cdot\xi_j)$ for all $s$.

The result of the above analysis is that, we can assign any $v_j \in \text{ker}(-\tau_j - \alpha \cdot \xi_j)$ as the initial data for $\sigma(w_j)$ at $\Gamma_j(0) = (y_j, \eta_j)$, by choosing initial data $f_j$ such that $C\Tilde{p}\sigma(f_j)(\Gamma_j(\frac{T'}{2\tau_j})) = \exp(\frac{T'}{2\tau_j}A_0)v_j$.

According to Theorem \ref{transport equation for principal symbol}, $\sigma(w_j)$ along the bicharacteristic $\Gamma_j$ solves the ODE
\begin{equation}\label{first order linearization ode}
    \begin{cases}
        \frac{d}{ds} \sigma(w_j)(\Gamma_j(s)) = A(\Gamma_j(s))\sigma(w_j)(\Gamma_j(s)) \text{ for } 0<s<\frac{T}{2\tau_j}\\
        \sigma(w_j)(\Gamma_j(0)) = v_j \in \text{ker}(-\tau_j - \alpha\cdot\xi_j)
    \end{cases}
\end{equation}
where
\begin{equation*}
    A(y, \eta) = i(\tau - \alpha \cdot \xi)(\beta+F'(u(y))).
\end{equation*}
Similar argument can also show that the initial data to end data map, $v_j \mapsto \sigma(w_j)(y_j^T, \eta_j)$, is bijective on $\text{ker}(-\tau_j - \alpha \cdot \xi_j)$.

From Proposition \ref{solution map prop}, $w_j(T) = \partial_{\epsilon_j}|_{\epsilon = 0} \mathcal{S}(\phi + \sum \epsilon_j \phi_j)$, again use the fact that $WF(w_j) \subset \Lambda_j$, we have
\begin{equation*}
    \sigma(w_j)(y_j^T, \eta_j) = \sigma(\mathcal{R}_T)(x_j^T, \xi_j; y_j^T, \eta_j)^{-1}\sigma(\partial_{\epsilon_j}|_{\epsilon = 0} \mathcal{S}(\phi + \sum \epsilon_j \phi_j))(x_j^T, \xi_j)
\end{equation*}
is determined by the solution map. Note that all the computations hold when $F$ also depends on $x$. In conclusion, we have the following result.
\begin{lemma}\label{solution map determines initial data to final data}
    For $i= 1, 2$, consider the ODE
    \begin{equation*}
        \begin{cases}
            \frac{d}{ds} w_i = A_i(\Gamma_j(s))w_i \text{ for } 0<s<\frac{T}{2\tau_j}\\
            w_i(0) = v
        \end{cases}
    \end{equation*}
    where
    \begin{equation*}
        A_i(y, \eta) = i(\tau - \alpha \cdot \xi)(\beta + F_i'(x, u(y))).
    \end{equation*}
    Then the initial data to final data map
    \begin{align*}
        W_i: & \text{ker}(-\tau_j - \alpha\cdot\xi_j) \xrightarrow{} \text{ker}(-\tau_j - \alpha\cdot\xi_j),\\
        &v \mapsto w_i(\frac{T}{2\tau_j})
    \end{align*}
    is bijective. Moreover, the same assumption as in Theorem \ref{main theorem} implies $W_1 = W_2$.
\end{lemma}
Note that $W_i$ is indeed a bijection on the entire $\mathbb{C}^4$, but as we can only set up initial data in $\text{ker}(-\tau_j - \alpha\cdot\xi_j)$, we can only say the two $W_i$ agree on that part. Finally, since we chose $f_j$ to be a classical conormal distribution, $\sigma(w_j)$ is also going to be homogeneous because of \eqref{distorted wave symbol equation} and \eqref{first order linearization ode}.

%%%%%%%%%%%%%%%%%%%%%%%%%%%%%%%%%%%%%%%%%%%%%%%%%%%%%%%%%%%%%%%%%%%%%%%
\section{Third order linearization}\label{section: third order linearization}
Let $\Gamma_1, \Gamma_2, \Gamma_3$ intersect at a point $y_c$, and let $K_1, K_2, K_3$ intersect transversally. Choose some future pointed light-like direction $\eta_0 = \sum_{j=1}^3 k_j \eta_j \in \Lambda_{123}\backslash\cup_{j=1}^3\Lambda_j$, $k_j \neq 0$. We shall compute the principal symbol of $w_{12}$ and $w_{123}$, the result for other $w_{\alpha}$'s are similar. According to the asymptotic analysis we have
\begin{align*}
    w_{12} = &Q(2F^{(2)}(u, w_1, w_2))\\
    \begin{split}
        w_{123} = &Q(6F^{(3)}(u, w_1, w_2, w_3) + 2F^{(2)}(u, w_{12}, w_3) \\
        &+ 2F^{(2)}(u, w_{13}, w_2) + 2F^{(2)}(u, w_{23}, w_1)).
    \end{split}
\end{align*}
Note that even though $w_j$'s are vectors of distributions, each coordinate of $F^{(2)}(u, w_1, w_2)$ and other similar terms are just linear combinations of multiplication of distributions, hence Lemma \ref{multiplication 1} and Lemma \ref{multiplication 2} can be used directly.

We start by analyzing $w_{12}$. By Lemma \ref{multiplication 1}, 
\begin{equation*}
    F^{(2)}(u, w_1, w_2) \in \mathcal{I}^{\mu, \mu+1}(\Lambda_{12}, \Lambda_1)+\mathcal{I}^{\mu, \mu+1}(\Lambda_{12}, \Lambda_2),
\end{equation*}
and by Proposition \ref{causal inverse 2}
\begin{equation*}
    w_{12} \in \mathcal{I}^{\mu, \mu}(\Lambda_{12}, \Lambda_1) + \mathcal{I}^{\mu, \mu}(\Lambda_{12}, \Lambda_2).
\end{equation*}
Moreover for $(y, \zeta) \in \Lambda_{12} \backslash (\Lambda_1 \cup \Lambda_2)$
\begin{equation*}
    \sigma_{\Lambda_{12}}(w_{12})(y, \zeta) = \pi^{-1} p^{-1}(\eta) F^{(2)}(u(y), \sigma(w_1)(y, \zeta_1), \sigma(w_2)(y, \zeta_2))
\end{equation*}
where $\zeta = \zeta_1+\zeta_2$ for $\zeta_j \in N^*_yK_j$.

Now we can analyze $w_{123}$. By Lemma \ref{multiplication 2}, on $\Lambda_{123}\backslash\cup_{j=1}^3 \Lambda_j$
\begin{equation*}
    F^{(2)}(u, w_{12}, w_3) \in \mathcal{I}^{3\mu-1/2}(\Lambda_{123})
\end{equation*}
and the principal symbol satisfies
\begin{equation*}
    \begin{split}
        &\sigma_{\Lambda_{123}}(F^{(2)}(u, w_{12}, w_3)) = (2\pi)^{-2}F^{(2)}(u(y^c), \\
        &p^{-1}(k_1\eta_1 + k_2\eta_2) F^{(2)}(u(y^c), \sigma(w_1)(y^c, k_1\eta_1), \sigma(w_2)(y^c, k_2\eta_2)), \sigma(w_3)(y^c, k_3\eta_3)).
    \end{split}
\end{equation*}
On the other hand, by Lemma \ref{multiplication 1} and Lemma \ref{multiplication 2}, on $\Lambda_{123}\backslash\cup_{j=1}^3 \Lambda_j$ we have
\begin{equation*}
    F^{(3)}(u, w_1, w_2, w_3) \in \mathcal{I}^{3\mu+1/2}(\Lambda_{123}),
\end{equation*}
and the principal symbol satisfies
\begin{equation*}
    \begin{split}
        \sigma_{\Lambda_{123}}(F^{(3)}(u, w_1, w_2, w_3))(y^c, \eta_0) = & (2\pi)^{-2}F^{(3)}(u(y^c), \sigma(w_1)(y^c, k_1\eta_1),\\
        &\sigma(w_2)(y^c, k_2\eta_2), \sigma(w_3)(y^c, k_3\eta_3)).
    \end{split}
\end{equation*}
Hence the top order singularity for the right hand side of
\begin{equation*}
    \begin{cases}
        \begin{split}
            Pw_{123} &= 6F^{(3)}(u, w_1, w_2, w_3) + 2F^{(2)}(u, w_{12}, w_3) \\
            &+ 2F^{(2)}(u, w_{13}, w_2) + 2F^{(2)}(u, w_{23}, w_1)
        \end{split}\\
        w_{123}(0) = 0
    \end{cases}
\end{equation*}
is given by $6F^{(3)}(u, w_1, w_2, w_3)$. By Proposition \ref{causal inverse 1} we have
\begin{equation*}
    \begin{cases}
        (\mathcal{L}_{H_q}+i\Tilde{p}p^s)\sigma(w_{123}) = 0 \text{ along } \Gamma_0 \\
        \sigma(w_{123})(y^c, \eta_0) = C\Tilde{p}\sigma_{\Lambda_{123}}(F^{(3)}(u, w_1, w_2, w_3))(y^c, \eta_0).
    \end{cases}
\end{equation*}
Similar to the first order linearization it is an ODE of the form
\begin{equation}\label{third order linearization ode}
    \begin{cases}
        \frac{d}{ds} \sigma(w_{123})(\Gamma_0(s)) = A(\Gamma_0(s))\sigma(w_{123})(\Gamma_0(s)) \text{ for } c<s<\frac{T}{2\tau_0}\\
        \sigma(w_{123})(\Gamma_0(c)) = C\Tilde{p}\sigma_{\Lambda_{123}}(F^{(3)}(u, w_1, w_2, w_3))(\Gamma_0(c))
    \end{cases}
\end{equation}
where $c$ is the time such that $\Gamma_0(c) = (y_c, \eta_0)$. Again by Proposition \ref{solution map prop} and same argument as $\mathcal{R}_0$, we have that $\sigma(\mathcal{R}_T)(x_0, \xi_0; y_0^T, \eta_0) \neq 0$ and
\begin{equation}\label{third order linearization solution map}
    \sigma(w_{123})(y_0^T, \eta_0) = \sigma(\mathcal{R}_T)(x_0, \xi_0; y_0^T, \eta_0)^{-1}\sigma(\partial_{\epsilon_{123}}|_{\epsilon = 0} \mathcal{S}(\phi + \sum \epsilon_j \phi_j))(x_0, \xi_0),
\end{equation}
so $\sigma(w_{123})(y_0^T, \eta_0)$ is determined by the solution map.

%%%%%%%%%%%%%%%%%%%%%%%%%%%%%%%%%%%%%%%%%%%%%%%%%%%%%%%%%%%%%%%%%%%%%%%
\section{Limit of collisions to boundary}\label{section: limit of collisions to boundary}
We now construct a sequence of collisions along a fixed bicharacteristic and let the collision point getting closer and closer to time 0. First fix some $\eta_j = (1, \xi_j)$, $j = 0, 1, 2, 3$, future pointing light-like such that $\eta_0 = \sum k_j\eta_j$ with $k_j \neq 0$. Let
\begin{equation*}
    \Gamma_0(s) = (2s, x_0 - 2s\xi_0, 1, \xi_0) 
\end{equation*}
travel from $(y_0, \eta_0) = (0, x_0, 1, \xi_0)$ to $(y_0^T, \eta_0) = (T, x_0^T, 1, \xi_0)$, and denote $y^c = (2c, x_0 - 2c\xi_0)$ the point at time $s=c$. Similarly for $j=1,2,3$ denote
\begin{equation*}
    \Gamma_j^c(s) = (2s, x_j^c - 2s\xi_j, 1, \xi_j) 
\end{equation*}
such that at some time $T_j^c$, $\Gamma_j^c(T_j^c) = (y^c, \eta_j)$, that is the projection is a geodesic from $y_j^c = (0, x_j^c)$ to $y^c$, see Figure \ref{fig:graph 1}.

\begin{figure}[htbp] % Position: here, top, bottom, page
    \centering
    % Insert image with adjusted width (e.g., 0.8 times the text width)
    \includegraphics[width=0.8\textwidth]{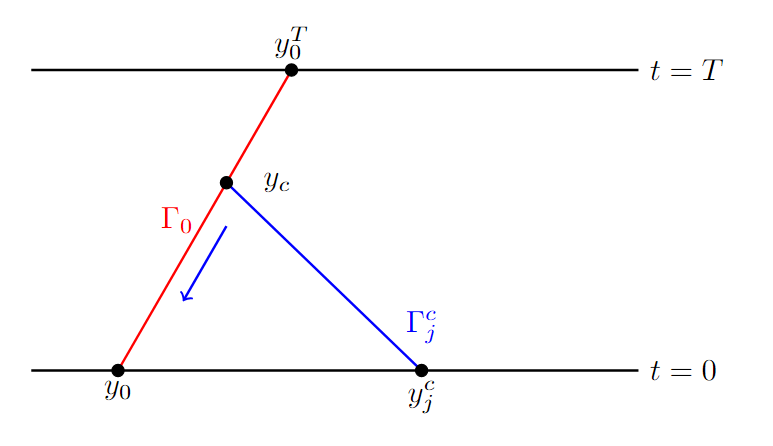} 
    \caption{$\Gamma_j^c$ intersects $\Gamma_0$ at point $y_c$. As $c \mapsto 0$ the collision point approaches to $y_0$, and $\Gamma_j^c$ goes to a single point.}
    \label{fig:graph 1}
\end{figure}

Note that $\Gamma_j^c(c) = (y^c, \eta_j)$, and as $c \xrightarrow{} 0$ we have
\begin{equation*}
    y^c \xrightarrow{} y_0, \quad y_j^c \xrightarrow{} y_0,
\end{equation*}
and $\Gamma_j^c$ becomes just the point $y_0$. Intuitively, if we consider the scenario to be that $\Gamma_1^c, \Gamma_2^c, \Gamma_3^c$ intersect at $y^c$ and leaves the point along $\Gamma_0$ from $y^c$ to $y_0^T$, then the limiting case is just collision at the boundary point $y_0$. For each collision point $y^c$, we will construct corresponding initial data $\phi_j^c$, and so the solutions now depend on $c$, that is the linearization terms are now $w_j^c$, $w_{ij}^c$ and $w_{123}^c$.

Using result from first order linearization \eqref{first order linearization ode}, we can setup initial data such that $\sigma(w_j^c)$ satisfies the ODE
\begin{equation*}
    \begin{cases}
        \frac{d}{ds} \sigma(w_j^c)(\Gamma_j^c(s)) = A(\Gamma_j^c(s))\sigma(w_j^c)(\Gamma_j^c(s)) \text{ for } 0<s<c\\
        \sigma(w_j^c)(\Gamma_j^c(0)) = v_j \text{ in } \text{ker}(-\tau_j - \alpha\cdot\xi_j).
    \end{cases}
\end{equation*}
We want to show that as $c \xrightarrow{} 0$, $\sigma(w_j^c)(\Gamma_j^c(c)) \xrightarrow{} v_j$.
When $c$ is small we have $\Gamma_j^c$ is entirely in a neighborhood of $y_0$, and as $F'(u)$ is smooth, we have for the max matrix norm $|A(\Gamma_j^c(s))| \leq N$ for all $0 \leq s \leq c$ for some constant $N$. Note that
\begin{align*}
    |\sigma(w_j^c)(\Gamma_j^c(s)) - v_j| &= |\int_0^c \frac{d}{ds}(\sigma(w_j^c)(\Gamma_j^c(s')))ds'|\\
    &= |\int_0^c A(\Gamma_j^c(s'))\sigma(w_j^c)(\Gamma_j^c(s'))ds'|\\
    &\leq N\int_0^c|\sigma(w_j^c)(\Gamma_j^c(s'))|ds'.
\end{align*}
By Grönwall's inequality
\begin{equation*}
    |\frac{d}{ds}\sigma(w_j^c)(\Gamma_j^c(s))| \leq N|\sigma(w_j^c)(\Gamma_j^c(s))| \implies |\sigma(w_j^c)(\Gamma_j^c(s))| \leq e^{Ns}|v_j|.
\end{equation*}
Hence
\begin{equation*}
    |\sigma(w_j^c)(\Gamma_j^c(s)) - v_j| \leq (e^{Nc} - 1)|v_j|
\end{equation*}
which converges to 0 as $c \xrightarrow{} 0$.

Similarly the third order linearization tells us that $\sigma(w_{123}^c)$ satisfies
\begin{equation*}
    \begin{cases}
        (\mathcal{L}_{H_q}+i\Tilde{p}p^s)\sigma(w_{123}^c) = 0 \text{ along } \Gamma_0 \\
        \sigma(w_{123}^c)(y^c, \eta_0) = C\Tilde{p}\sigma_{\Lambda_{123}}(F^{(3)}(u, w_1^c, w_2^c, w_3^c))(y^c, \eta_0).
    \end{cases}
\end{equation*}
Again we can write it as an ODE like \eqref{third order linearization ode} and denote $w^c(s) = \sigma(w_{123}^c)(\Gamma_0(s))$ for simplicity
\begin{equation}\label{ODE w^c}
    \begin{cases}
        \frac{d}{ds} w^c(s) = A(\Gamma_0(s))w^c(s) \text{ for } c < s < T/2 \\
        \begin{split}
            w^c(c) = C'(1-\alpha\cdot\xi_0)F^{(3)}&(u(y^c), \sigma(w_1^c)(\Gamma_1^c(c)),\\
            &\sigma(w_2^c)(\Gamma_2^c(c)), \sigma(w_3^c)(\Gamma_3^c(c))),
        \end{split}
    \end{cases}
\end{equation}
where $C' = C \prod_{j=1}^3 k_j^{\mu+1/2} \neq 0$ because $\sigma(w_j^c)$ is homogeneous. Now consider another ODE
\begin{equation}\label{ODE limiting case}
    \begin{cases}
        \frac{d}{ds} w(s) = A(\Gamma_0(s))w(s) \text{ for } 0<s<T/2 \\
        w(0) = C'(1-\alpha\cdot\xi_0)F^{(3)}(\phi(x_0), v_1, v_2, v_3).
    \end{cases}
\end{equation}
We want to show that $w^c(T/2) \xrightarrow{} w(T/2)$ as $c \xrightarrow{} 0$. First of all, by previous computations it is obvious that as $c \xrightarrow{} 0$,
\begin{equation*}
    w^c(c) \xrightarrow{} w(0).
\end{equation*}
Note that $w$ and $w^c$ essentially solve the same ODE with different initial data at time $c$, where
\begin{equation*}
    w(c) = w(0) + \int_0^c w'(s)ds = w(0) + \int_0^c A(\Gamma_0(s))w(s) ds.
\end{equation*}
In a neighborhood of $\Gamma_0$ we have $|A(\Gamma_0(s))| \leq \Tilde{N}$ since it is smooth. Thus use Grönwall's inequality again we obtain
\begin{align*}
    |w(c) - w^c(c)| &\leq |w(c) - w(0)| + |w(0)-w^c(c)| \\
    &\leq (e^{\Tilde{N}c} - 1)|w(0)| + |w(0) - w^c(c)| \\
    &\xrightarrow{} 0
\end{align*}
as $c \xrightarrow{} 0$. Hence
\begin{align*}
    |w(T/2) - w^c(T/2)| &\leq |w(c) - w^c(c)| + |\int_c^{T/2} A(\Gamma_0(s))(w(s) - w^c(s))ds|\\
    &\leq |w(c) - w^c(c)| + (e^{MT/2} - e^{\Tilde{N}c})|w(c) - w^c(c)|\\
    &\xrightarrow{} 0
\end{align*}
as $c \xrightarrow{} 0$. By \eqref{third order linearization solution map},
\begin{equation*}
    w(T/2) = \lim_{c \xrightarrow{} 0} \sigma(\mathcal{R}_T)(x_0, \xi_0; y_0^T, \eta_0)^{-1}\sigma(\partial_{\epsilon_{123}}|_{\epsilon = 0} \mathcal{S}(\phi + \sum \epsilon_j \phi_j^c))(x_0, \xi_0)
\end{equation*}
meaning the solution of \eqref{ODE limiting case} at time $T/2$ can be determined by the solution map. Finally, note that \eqref{ODE limiting case} is the same ODE as the one in Lemma \ref{solution map determines initial data to final data}, the initial data can also be determined by the injectivity result proved there. Again all the computations are the same when $F$ also depends on $x$. For any $z < M$ we can find some $\phi \in H^s_{\delta} \cap C^{\infty}(\mathbb{R}^3)$ such that $\phi(x_0) = z$, so we obtain the following lemma.

\begin{lemma}\label{lemma 6.1}
    For $j=1,2,3$, let $\xi_j$ be linearly independent lightlike covectors. Under the same assumption as in Theorem \ref{main theorem}, we have
    \begin{equation*}
        (1-\alpha\cdot\xi_0)\partial_z^3F_1(x, z, v_1, v_2, v_3) = (1-\alpha\cdot\xi_0)\partial_z^3F_2(x, z, v_1, v_2, v_3)
    \end{equation*}
    for all $v_j \in \text{ker}(-1 - \alpha \cdot \xi_j)$, $\xi_0 = \sum_{j=1}^3 k_j\xi_j$ where $k_j \neq 0$, $x \in \mathbb{R}^3$ and $|z| \leq M$.
\end{lemma}

%%%%%%%%%%%%%%%%%%%%%%%%%%%%%%%%%%%%%%%%%%%%%%%%%%%%%%%%%%%%%%%%%%%%%%%
\section{Proof of main theorem}\label{section: proof of main theorem}
Here is a linear algebra lemma that will be used later.

\begin{lemma}\label{linear algebra lemma}
        Let $K_1, K_2$ be some linear subspace of $\mathbb{C}^n$ passing through the origin such that $K_1 \cap K_2 = \{0\}$, then for any $v_1, v_2 \in \mathbb{C}^n$, $v_1 + K_1 \cap v_2 + K_2$ has at most 1 element.
\end{lemma}
\begin{proof}
    Suppose $w, w' \in v_1 + K_1 \cap v_2 + K_2$. Since $K_1$ is a linear subspace, $w, w' \in v_1 + K_1$ implies $w = v_1 + k_1$ and $w' = v_1 + k_2$ where $k_1, k_2 \in K$, hence $w - w' = k_1-k_2 \in K_1$. Similarly, $w - w' \in K_2$. Thus $w = w'$ because $K_1 \cap K_2 = \{0\}$.
\end{proof}
Now we are ready to prove Theorem \ref{main theorem}. We continue to use $F^{(k)}(x, z)$ to denote derivatives with respect to $z$.
\begin{proof}
    By Lemma \ref{lemma 6.1}, for linearly independent lightlike covectors $\eta_j = (1, \xi_j)$, $j = 1, 2, 3$, and $\eta_0 = \sum k_j \eta_j$ with $k_j \neq 0$, we can determine
    \begin{equation*}
        (1-\alpha\cdot\xi_0)F^{(3)}(x, z, v_1, v_2, v_3)
    \end{equation*}
    for $v_j \in \text{ker}(-1 - \alpha \cdot \xi_j)$, $x \in \mathbb{R}^3$, $|z| \leq M$. Since $v \in \text{ker}(1-\alpha \cdot \xi_0)^{\perp} \iff v \in \text{ker}(1+\alpha \cdot \xi_0) \iff (1-\alpha\cdot\xi_0)v = 2v$,
    \begin{equation*}
        F^{(3)}(x, z, v_1, v_2, v_3) \in \frac{1}{2}(1-\alpha\cdot\xi_0)F^{(3)}(x,z, v_1, v_2, v_3) + \text{ker}(1-\alpha\cdot\xi_0).
    \end{equation*}
    Similarly use some other appropriate $\eta_0'$ instead of $\eta_0$ we can determine $(1-\alpha\cdot\xi_0')F^{(3)}(x, z, v_1, v_2, v_3)$ and hence
    \begin{equation*}
        F^{(3)}(x, z, v_1, v_2, v_3) \in \frac{1}{2}(1-\alpha\cdot\xi_0')F^{(3)}(x, z, v_1, v_2, v_3) + \text{ker}(1-\alpha\cdot\xi_0').
    \end{equation*}

    \begin{figure}[htbp] % Position: here, top, bottom, page
        \centering
        % Insert image with adjusted width (e.g., 0.8 times the text width)
        \includegraphics[width=0.8\textwidth]{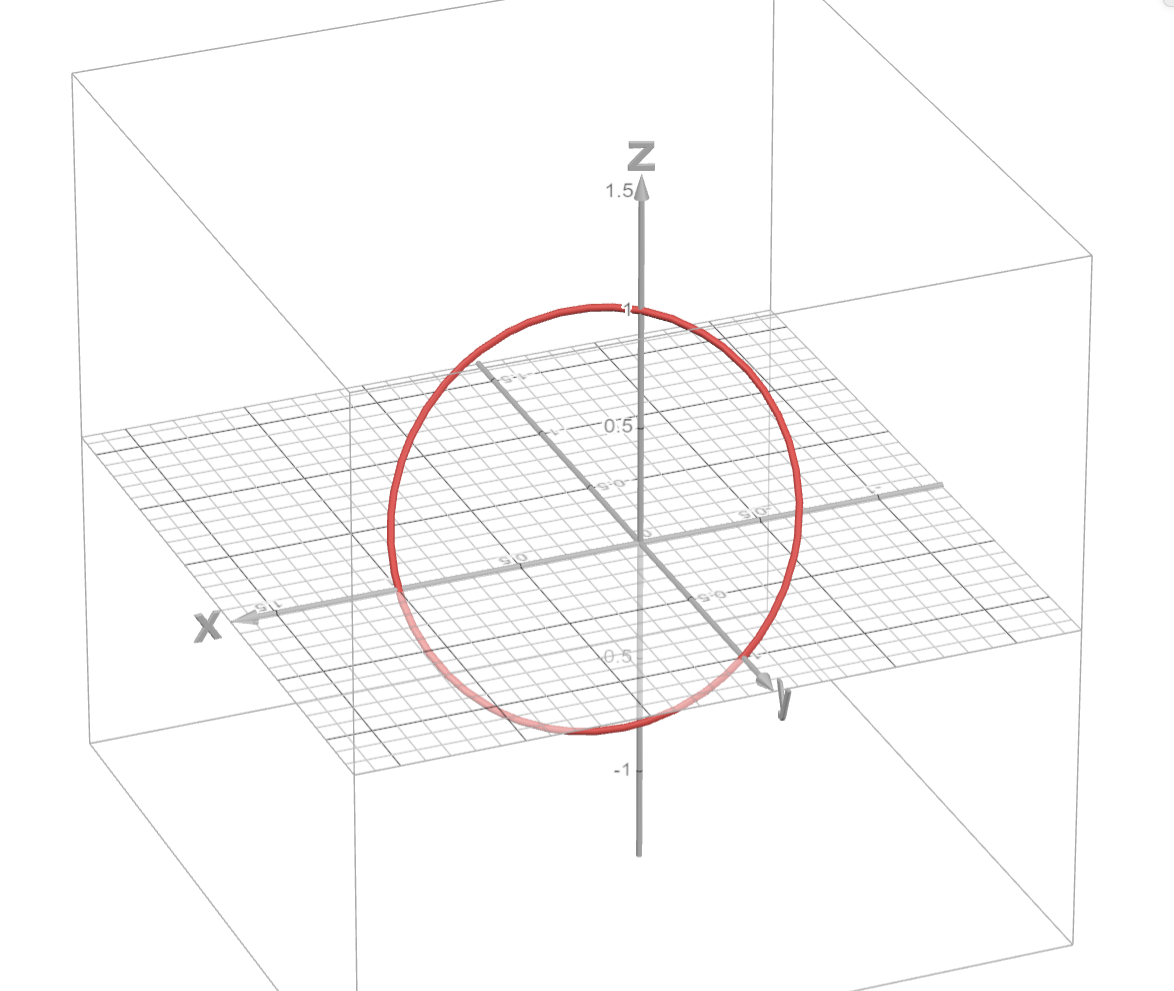} 
        \caption{The set of available directions for $\eta_0$ when $\eta_j$ are given by $(1, e_j)$.}
        \label{fig:graph 2}
    \end{figure}

    We now choose the appropriate $\eta_1, \eta_2, \eta_3$ and $\eta_0$. The principal symbol of $P$ at $(t, x, \tau, \xi) \in L\mathbb{R}^4$ (we shall omit $(t, x)$ since principal symbol doesn't depend on them) for $\xi = (\xi^1, \xi^2, \xi^3)$ is
    \begin{equation*}
        p(\tau, \xi)=\begin{pmatrix}
            -\tau & 0 & -\xi^3 & -\xi^1+i\xi^2 \\
            0 & -\tau & -\xi^1-i\xi^2 & \xi^3 \\
            -\xi^3 & -\xi^1+i\xi^2 & -\tau & 0\\
            -\xi^1-i\xi^2 & \xi^3 & 0 & -\tau
        \end{pmatrix}
    \end{equation*}
    and the kernel is given by
    \begin{equation*}
        \text{ker}(p(\tau, \xi)) = \text{span}\{(\xi^3, \xi^1+i\xi^2, -\tau, 0), (-\xi^1+i\xi^2, \xi^3, 0, \tau)\}.
    \end{equation*}
    Consider now $\eta_1 = (1, 1, 0, 0), \eta_2 = (1, 0, 1, 0), \eta_3 = (1, 0, 0, 1)$. Then
    \begin{equation*}
        \mathcal{M} = \{(1, a, b, c): a+b+c = a^2+b^2+c^2=1, a, b, c \neq 1\}
    \end{equation*}
    forms a 1-dimensional set of available directions for $\eta_0$, see Figure \ref{fig:graph 2}. For example, we can take $\eta_0 = (1, -\frac{1}{3}, \frac{2}{3}, \frac{2}{3})$, $\eta_0' = (1, \frac{2}{3}, -\frac{1}{3}, \frac{2}{3})$, and use the fact that $1-\alpha\cdot\xi = p(-1, \xi)$ to obtain
    \begin{align*}
        &\text{ker}(1-\alpha\cdot\xi_0) = \text{span}\{(2, -1+2i, 3, 0), (1+2i, 2, 0, -3)\} \\
        &\text{ker}(1-\alpha\cdot\xi_0') = \text{span}\{(2, 2-i, 3, 0), (-2-i, 2, 0, -3)\}
    \end{align*}
    which have trivial intersection, and hence by Lemma \ref{linear algebra lemma} one can determine $F^{(3)}(x, z, v_1, v_2, v_3)$.
    
    So far, we are able to determine any $F^{(3)}(x, z, v_1, v_2, v_3)$ for $v_j \in \text{ker}(p(\eta_j))$ where $\eta_j = (1,e_j)$ and $e_1 = (1, 0, 0), e_2 = (0, 1, 0), e_3 = (0, 0, 1)$. On the other hand, the same thing can be done if one replaces some of the $\eta_j$ with $\eta_j' = (1, -e_j)$. That is, we are able to determine any $F^{(3)}(x, z, v_1, v_2, v_3)$ for $v_j  \in \text{ker}(p(\eta_j))\cup\text{ker}(p(\eta_j'))$. Straightforward computation shows that $\text{ker}(p(\eta_j))+\text{ker}(p(\eta_j')) = \mathbb{C}^4$, thus by linearity we can determine $F^{(3)}(x, z, v_1, v_2, v_3)$ for any $v_j \in \mathbb{C}^4$, $x \in \mathbb{R}^3, |z| \leq M$.
    
    If we have the additional assumptions that $F'(x, 0)$ and $F^{(2)}(x, 0)$ are already determined, then along with the fact that $F(x, 0) = 0$, $F(x, z)$ can be determined for any $x \in \mathbb{R}^3, |z| \leq M$.
\end{proof}

\section*{Declarations}
The author states that there is no conflict of interest, and there is no associated data used in the paper.

\printbibliography
\end{document}